\documentclass[11pt]{amsart}

\usepackage{fullpage,url,amssymb,enumerate,colonequals}
\usepackage{url,amssymb,enumerate,colonequals}

\usepackage{mathrsfs} 
\usepackage[section]{placeins}
\usepackage{MnSymbol}
\usepackage{extarrows}
\usepackage{lscape}
\usepackage[all,cmtip]{xy}

\usepackage[OT2,T1]{fontenc}
\usepackage{color}

\usepackage[
        colorlinks, citecolor=darkgreen,
        backref,
        pdfauthor={Elvira Lupoian},
]{hyperref}

\usepackage{amsmath}
\usepackage{amsthm}
\usepackage{amssymb}
\usepackage{tikz-cd}
\usepackage{tikz}
\usepackage{tabularx}
\usepackage{longtable}

\newcommand{\Z}{\mathbb{Z}}

\DeclareMathOperator{\divv}{div}

\DeclareMathOperator{\Gal}{Gal}

\newcommand{\GL}{\operatorname{GL}}

\newcommand{\SL}{\operatorname{SL}}

\numberwithin{equation}{section}

\newtheorem{lemma}{Lemma}

\newtheorem{proposition}{Proposition}

\theoremstyle{definition}

\newtheorem*{Ack}{Acknowledgements}
\theoremstyle{remark}
\newtheorem{remark}[equation]{Remark}

\theoremstyle{remark}
\newtheorem{case1}{Case}
\newtheorem{case2}{Case}

\newenvironment{psmallmatrix}
  {\left(\begin{smallmatrix}}
  {\end{smallmatrix}\right)}

\definecolor{darkgreen}{rgb}{0,0.5,0}

\DeclareRobustCommand{\SkipTocEntry}[5]{}

\begin{document}

\title{Three-torsion subgroups and wild conductors of genus 3 hyperelliptic curves }

\begin{abstract}
We give a practical method for computing the 3-torsion subgroup of the Jacobian of a genus 3 hyperelliptic curve. We define a scheme for the 3-torsion points of the Jacobian and use complex approximations, homotopy continuation and lattice reduction to find precise expression for the 3-torsion. In the latter stages of the paper we explain how the 3-torsion subgroup can be used to compute the wild part of the local exponent of the conductor at 2. 
\end{abstract}

\author{Elvira Lupoian}

\address{Mathematics Institute\\
    University of Warwick\\
    CV4 7AL \\
    United Kingdom}

\email{elupoian@gmail.com}

\date{\today}
\thanks{During the completion of this work the author was supported by the EPSRC studentship EP/V520226/1}

\maketitle

\section{Introduction}

Let $C$ be a smooth and projective curve of genus $g$ defined over a number field $K$. The Jacobian $J$ of $C$ is a $g$-dimensional abelian variety whose points can be identified with linear equivalence classes of degree $0$ divisors on the curve. If the $K$-rational points of $C$ are know, one can easily construct $K$-rational points on the Jacobian. However, it's not clear how points on the Jacobian can be constructed without prior knowledge of points on the curve. A natural question to ask is whether one can describe and compute torsion points on the Jacobian, that is, linear equivalence classes of degree $0$ divisors $D$ such that $N \cdot D = \divv(h)$ for some function on the curve $h$ and $N \in \mathbb{N}_{>0}$. 

When $C$ is an elliptic curve, describing $N$-torsion points is straightforward due to the group law operation and the explicit formula that one can express for addition. For hyperelliptic curves, $2$-torsion points are linear combinations of Weierstrass point, see \cite{HyperCurves} or \cite{HyperCurves2}. Moreover, computing such points reduces to factorising the defining polynomial of the hyperelliptic curve. For non-hyperelliptic curves, the group of $2$-torsion points of the Jacobian is generated by differences of odd theta characteristics (see \cite[Chapter 5]{dolgachev2012classical}), which in turn correspond to multi-tangent hyperplanes to the curve. Computing the $2$-torsion subgroup in the case of genus $3$, $4$ and $5$ curves was previously considered by the author in \cite{2torsion}. Torsion points of higher order on genus $2$ are explicitly described by Flynn, Testa and Bruin \cite{bruin2014descent} and Flynn \cite{flynn2015descent}, and used for descent purposes. An adjacent problem is that of finding high order rational torsion points on Jacobians of curves, and a number of works have made significant progress towards this, see for instance \cite{bernard2009jacobians}, \cite{howe2000large}, \cite{flynn1990large}, \cite{leprevost1992torsion}, \cite{howe2015genus}. There have also been developments in computing the entire rational torsion subgroup of Jacobians, by Stoll \cite{stoll1998height} in the case of genus $2$ curves and  by Muller and Reistma \cite{j2023computing} in the case of hyperelliptic curves of genus $3$.

In this paper we consider hyperelliptic curve of genus $3$. We give a general description of $3-$torsion points on Jacobians of such curves and a method for computing them. More specifically, suppose that 
\begin{center}
$C : y^{2} = f\left( x \right)$
\end{center}
is an affine model of the curve, where $f \left( x \right) \in K[x ] $  has degree 7 or 8, and no repeated roots. In Section \ref{schm} we show that when $\deg(f) =7$, $3$-torsion points correspond to $\alpha_{1}, \ldots, \alpha_{10}$ such that 
\begin{center}
$f\left( x \right) \left( x + \alpha_{1} \right)^{2} + \alpha_{7} \left( x^{3} + \alpha_{8}x^{2} + \alpha_{9}x + \alpha_{10} \right)^{3} = \left( \alpha_{2}x^{4} + \alpha_{3}x^{3} + \alpha_{4}x^{2} + \alpha_{5}x + \alpha_{6} \right)^{2}$
\end{center}
and when $\deg(f) =8$ they correspond to solutions of 
\begin{align*}
&  \left( -x^{6} - \frac{a_{7}}{2}x^{5}  - \left( -\frac{a_{6}}{2} + \frac{a_{7}^{2}}{8} \right)x^{4} + \alpha_{1} \left( - x^{5} - \frac{a_{7}}{2}x^{4} \right) - \alpha_{2} x^{4} + \alpha_{3}x^{3} + \alpha_{4}x^{2} + \alpha_{5}x + \alpha_{6} \right)^{2}  \\ &  \ \ \ \ \ \ =  \alpha_{7} \left( x^{3} + \alpha_{8}x^{2} \alpha_{9}x + \alpha_{10} \right)^{3}  + \left( x^{2} + \alpha_{1}x + \alpha_{2}\right)^{2}  f\left( x \right)
\end{align*}
where the $a_{i}$ are coefficients of $f$.

The above define a system of $10$ equations in $10$ variables, and computing the entire $3$-torsion subgroup $J[3]$ corresponds to computing the solutions of this system. Notably, this has $728$ solutions, which are generally defined over number fields of large degree. Therefore, classical Groebner basis techniques are not effective in computing the solutions of such a system. In Section \ref{comps}  we give a method for computing the points of such a zero dimensional scheme. This is a two step method; first we approximating the points as tuples of complex numbers using homotopy continuation and the Newton-Raphson method. The approximations are then used to find algebraic expressions using lattice reduction techniques. A similar method of complex approximations and lattice reduction was used by the author in \cite{2torsion} to compute the 2-torsion subgroup of some non-hyperelliptic modular Jacobians. In Section \ref{eg1}  we compute the 3-torsion subgroup of the Jacobian of the modular curve $X_{0}\left( 40 \right)$. Similar descriptions to above can be given more generally for hyperelliptic curves of higher genus and for torsion subgroup of other orders $N$. In theory, the general method describe in Section \ref{comps} can used to compute $N$-torsion subgroups of Jacobians of other curves. However, due the large degree, $N^{2g}$, this method becomes ineffective for general curves, for almost all orders $N$ and genera $g$. 

The second half of this paper explains how the 3-torsion subgroup $J\left[ 3\right]$  can be used to determine the local conductor exponent of $C$ at 2; this is a direct generalisation of \cite{DokDor}. The conductor of a curve $C/ \mathbb{Q}$ is a representation theoretic constant, defined as a product $N =  \prod_{p} p^{n_{p}}$ over the primes $p$ where $C$ has bad reduction. When $C$ is an elliptic curve, the local exponents $n_{p}$ can be computed using Tate's algorithm (see \cite[Chapter 4]{Silverman2}). For hyperelliptic curves of arbitrary genus,  there are formulae for $n_{p}$ for all $p \ne 2$, see \cite{Conductorexp}. 
For curves of genus $2$,  Dokchitser and Doris \cite{DokDor} give an algorithm for $n_{2}$. In particular, we note that $n_{2}$ is the sum of the tame and wild parts, 
\begin{center}
$n_{2} = n_{\text{tame}} + n_{\text{wild}}$
\end{center}
where $n_{\text{tame}}$ can be deduced from a regular model of the curve and $n_{\text{wild}}$ is the Swan conductor of the 3-adic Tate module of the Jacobian of $C$, and it can be computed from the action of $\Gal \left( K \left( J[3] \right) / K \right)$ on $J[3]$.

In Section \ref{cond} we give a brief theoretic overview of how local conductor exponents are calculated using a regular model of the curve and the torsion subgroups of the Jacobian.  

\begin{Ack}
The author is extremely thankful to  Samir Siksek and Damiano Testa for many helpful conversations and invaluable suggestions throughout this project. The author also sincerely thanks Tim Dokchitser for the helpful conversation regarding the tame part of the conductor; and the anonymous referee for their comments and corrections. During the completion of this work the author was a student at the University of Warwick and was supported by the EPSRC studentship EP/V520226/1
\end{Ack}

\section{Scheme of 3-torsion points }
\label{schm}
Let $C$ be a hyperelliptic curve of genus $g \ge 2$ defined over a number field $K$. More precisely, $C$ is a subvariety of the weighted projective space $\mathbb{P}{(1,g+1,1)}$ defined by an equation of the form $y^{2} = F(x,z)$, where $F(x,z) \in K[x,z]$ is a homogeneous polynomial of degree $2g+2$, which is square-free. Throughout this paper we work with an intersection of $C$
with affine patch of the form 
\begin{center}
$ y^{2} = f\left( x \right) := F(x,1) $
\end{center}
where $f\left( x \right) \in K\left[ x \right]$ has degree $2g+1$ or $2g+2$ and has no repeated roots. Moreover, after possible passing to a quadratic extension, we may assume $f$ to be monic.  We refer to the points of $C$ not appearing on the affine model as the \textit{points at infinity}.  These correspond to $Z=0$, and we observe that there is a single such point, namely $( 1 : 0 :0 ) $ when the degree of $f$ is $2g+1$; and $2$ points:  $( 1 : 1 :0)$ and $( 1 : -1:0)$ when the degree of $2g+2$ is 8. For the rest of this paper, a hyperelliptic curve will be interpreted as the affine curve defined by $y^{2} = f(x)$ along with points at infinity. See \cite{HyperCurves} or \cite{HyperCurves2} for details on the arithmetic of hyperelliptic curves.

Let $J$ be the Jacobian variety of $C$. This is a $g$-dimensional, abelian variety over $K$, whose points can be identified  with linear equivalence classes of degree $0$ divisors on $C$. 

The 3-torsion subgroup of $J$ consists of all elements $\left[ D \right]\in J $ such that $3D = \text{div}\left( h \right)$, where $h$ is a rational function on $C$.  To parametrise all such points in the case $g=3$, we treat the two degree cases separately. We begin with the following straightforward result, which is required throughout the remainder of the section.

\begin{lemma}
Let $C$ be a hyperelliptic curve of genus $g \ge 2$, defined over a number field $K$ and with affine model $y^{2} = f(x)$. Suppose that $g(x) \in K[x]$ is such that its divisor of zeros ( when $g(x)$ is viewed as element of $K(C)$) is of the form $3D$ for some effective divisor $D$ on $C$. Then $g(x) = h(x)^{3}$ for some $h \in \Bar{K}(C)$.
\begin{proof}
We can write $g$ as 
\begin{center}
    $g\left( x \right) = \alpha \left( x - \beta_{1} \right)^{r_{1}} \ldots \left( x - \beta_{s} \right)^{r_{s}}\left( x - \gamma_{1} \right)^{t_{1}} \ldots \left( x - \gamma_{n} \right)^{t_{n}} $
\end{center}
where $f\left( \beta_{i} \right) =0$ for all $i = 1 \ldots s$, $f \left( \gamma_{j} \right) \neq 0$ for all $j = 1 \ldots n$ and $\alpha \in K^{\times}$.The divisor of zero of $g$ is 
\begin{center}
    $\displaystyle \sum_{i=1}^{s} 2r_{i} \left( \beta_{i}, 0 \right) + \displaystyle \sum_{j=1}^{n} t_{j} \left(\left( \gamma_{j}, \sqrt{f\left(\gamma_{j} \right)} \right) + \left( \gamma_{j},  - \sqrt{f\left(\gamma_{j} \right)} \right)   \right) $
\end{center}
By assumption, this must equal $3D$, and hence $3$ divides $2r_{i}$ and $t_{j}$ for all $i = 1 \ldots s$ and $j = 1 \ldots n$, and the result follows.
\end{proof}
\end{lemma}
\begin{proposition}
Let $C$ be an odd degree hyperelliptic curve of genus $3$, over a number field $K$, with an affine model 
\begin{center}
$y^{2} = f\left( x \right) = x^{7} + a_{6}x^{6} + a_{5}x^{5} + a_{4}x^{4} + a_{3}x^{3} + a_{2}x^{2} + a_{1}x + a_{0} $
\end{center}
where $a_{i} \in K$ and $f$ has no repeated roots. Let $J$ be the Jacobian of $C$. Then any non-zero 3-torsion point of $J$ is the form $\left[ \frac{1}{3} \text{div} \left( h \right) \right]$ where  
\begin{center}
$h = y \left( x + \alpha_{1} \right) + \alpha_{2}x^{4} + \alpha_{3}x^{3} + \alpha_{4}x^{2} + \alpha_{5}x + \alpha_{6}$
\end{center}
with $ \alpha_{1}, \ldots, \alpha_{6}  \in \overline{K}$ satisfying
\begin{center}
$f\left( x \right) \left( x + \alpha_{1} \right)^{2} + \alpha_{7} \left( x^{3} + \alpha_{8}x^{2} + \alpha_{9}x + \alpha_{10} \right)^{3} = \left( \alpha_{2}x^{4} + \alpha_{3}x^{3} + \alpha_{4}x^{2} + \alpha_{5}x + \alpha_{6} \right)^{2}$
\end{center}
for some $ \alpha_{7}, \alpha_{8}, \alpha_{9}, \alpha_{10} \in \bar{K}$.
\end{proposition}

\begin{proof}
Let  $\infty$ be the unique point at infinity and  $\left[ D \right] \in J\left[ 3\right] \setminus 0$. By the Riemann-Roch theorem there exists a unique,  effective divisor $D_{0} = P_{1} + P_{2} + P_{3} $  such that 
\begin{center}
$ D \sim D_{0} - 3 \infty$
\end{center}
As $3D$ is principal,  $3D_{0} -9 \infty= \text{ div} \left( h \right) $ , where $h$ is a rational function on $C$. Thus $h$ is in the Riemann-Roch space $\mathcal{L} \left( 9 \infty \right)$. The Riemann-Roch theorem this space has dimension $7$. The point $\infty$ is the only pole of the functions $x$ and $y$, with $x$ having order exactly $2$ and $y$ having order exactly $7$ at the pole, hence 
\begin{center}
$ 1, x, x^{2}, x^{3}, x^{4}, xy, y$
\end{center}
is a basis for $\mathcal{L}(9 \infty)$. The functions above are linearly independent since they all have different orders of vanishing at the pole $\infty$. 

Replacing  $h$ by a scalar multiple if necessary,  $h$ is either a polynomial in $x$ of degree at most 4,  or $h = y + k\left(  x \right)$ with $k\left( x \right) \in K \left[ x \right]$, $\text{deg} \left( k \right) \leq 4$, or $h = y \left( x + \alpha_{1} \right) + k\left( x \right)$ with $\alpha_{1} \in K$, $k\left( x \right) \in K \left[ x \right]$ and $\text{deg} \left( k \right) \leq 4$.

\begin{case1}
Suppose $h \in K\left[ x \right]$ and $d = \text{deg}\left( h \right) \leq 4$. Let $\theta_{1}, \ldots , \theta_{d}$ be the roots of $h$.  The divisor of zeros of $h$ is $3D_{0} = 3P_{1} + 3P_{2} + 3P_{3}$ since $ \text{ div} \left( h \right) = 3D_{0} - 9 \infty$. We can also compute the divisor of zeros directly,  and find it to be 
\begin{center}
$\displaystyle \sum_{i=1}^{d} \left( \left( \theta_{i} , \sqrt{f\left( \theta_{i} \right)} \right) +  \left( \theta_{i} ,  - \sqrt{f\left( \theta_{i} \right)} \right) \right) $
\end{center}
The above divisor has degree at most $8$,  whilst $\text{deg}\left( 3D_{0} \right) =9$, and hence they cannot be equal. Thus $h$ cannot be a polynomial in $x$ of degree at most $4$.
\end{case1}
\begin{case1}
Suppose $h = y + g\left( x \right)$ where $g \in K \left[ x \right]$ and $\text{deg} \left( g \right) \leq 4$, and let $\tilde{h} = -y + g\left( x \right)$. As before, the divisor of zeros of $h$ is $3D_{0}$, and the divisor of zeros of $\tilde{h}$ is 
\begin{center}
$3 \iota \left( D_{0} \right) = 3\iota\left( P_{1} \right) + 3 \iota \left( P_{2} \right) + 3 \iota \left( P_{3} \right)$
\end{center}
where $\iota : C \longrightarrow C$ denotes the hyperelliptic involution on $C$. The divisor of zeros of $h\tilde{h}$ is $3D_{0} + 3\iota \left( D_{0} \right)$, and hence $h \tilde{h} = - f\left( x \right) + g \left( x \right)^{2}$ is necessarily a cube as an element of $\overline{K} \left[ x \right]$ by Lemma 1. However, this is a contradiction since $ - f\left( x \right) + g \left( x \right)^{2}$ has degree 7 or 8. 
\end{case1}
\begin{case1} This is the only remaining case.
Suppose $h = y \left( x + \alpha_{1} \right) + g \left( x \right)$ where $\alpha_{1} \in K $, $g\left( x \right) \in K \left[ x \right] $ and $g$ has degree at most 4,  and let $\tilde{h} = -  y \left( x + \alpha_{1} \right) + g \left( x \right)$. Arguing as before,  the divisor of zeros of $h \tilde{h}$ is $3D_{0} + 3\iota \left( D_{0} \right) $ and hence by Lemma 1, $h \tilde{h} \in K\left[ x \right]$ is necessarily a cube.  Hence 
\begin{align*}
 h \tilde{h} &= \left( y \left( x + \alpha_{1} \right) + g \left( x \right) \right) \left( -y \left( x + \alpha_{1} \right) + g \left( x \right) \right) \\
 &= -f\left( x \right) \left( x + \alpha_{1} \right)^{2} + g\left( x \right)^{2} \\
 &= \ \alpha_{7} \left( x^{3} + \alpha_{8} x^{2} + \alpha_{9} x + \alpha_{10} \right)^{3} 
\end{align*}
for some $\alpha_{7}, \ldots , \alpha_{10} \in \overline{K}$, where $g \left( x \right) = \alpha_{2} x^{4} + \alpha_{3} x^{3} + \alpha_{4}x^{2} + \alpha_{5}x + \alpha_{6}$ for some $\alpha_{2}, \ldots, \alpha_{6} \in \overline{K}$.  
\end{case1}
\end{proof}

Equating coefficients in this expression 
\begin{center}
$f\left( x \right) \left( x + \alpha_{1} \right)^{2} + \alpha_{7} \left( x^{3} + \alpha_{8}x^{2} + \alpha_{9}x + \alpha_{10} \right)^{3} = \left( \alpha_{2}x^{4} + \alpha_{3}x^{3} + \alpha_{4}x^{2} + \alpha_{5}x + \alpha_{6} \right)^{2}$
\end{center}
gives 10 equations in $\alpha_{1}, \ldots, \alpha_{10}$, where $\left( \alpha_{1}, \ldots, \alpha_{6} \right)$ define a 3-torsion point.  We refer to the scheme defined by these 10 equations as the scheme of 3-torsion points. 

\begin{proposition}
Let $C$ be an even degree hyperelliptic curve of genus $3$, over a number field $K$, with an affine model 
\begin{center}
$y^{2} = f\left( x \right) =x^{8} +  a_{7}x^{7} + a_{6}x^{6} + a_{5}x^{5} + a_{4}x^{4} + a_{3}x^{3} + a_{2}x^{2} + a_{1}x + a_{0} $
\end{center}
where $a_{i} \in K$ and $f$ has no repeated roots. Let $J$ be the Jacobian of $C$. Then any non-zero 3-torsion point of $J$ is the form $\left[ \frac{1}{3} \text{div} \left( h \right) \right]$ where  
\begin{align*}
 h  & =  x^{2}y - x^{6} -\frac{a_{7}}{2} x^{5} + \left( -\frac{a_{6}}{2} + \frac{a_{7}^{2}}{8} \right)x^{4} + \alpha_{1} \left( xy - x^{5} - \frac{a_{7}}{2} x^{4} \right) +   \alpha_{2} \left( y - x^{4} \right)  \\ &  \ \ \ \ \ + \alpha_{3}x^{3} + \alpha_{4} x^{2} + \alpha_{5} x + \alpha_{6} 
\end{align*}
for some $ \alpha_{1}, \ldots, \alpha_{6} \in \overline{K}$ satisfying 
\begin{center}
$-f \left( x \right) l \left( x \right)^{2} + g \left( x \right) ^{2} = \alpha_{7} \left( x^{3} + \alpha_{8}x^{2} + \alpha_{9}x + \alpha_{10} \right)^{3} $
\end{center}
for some $\alpha_{7},\ldots, \alpha_{10} \in \overline{K} $ where
\begin{align*}
l\left( x \right) &= x^{2} + \alpha_{1} x + \alpha_{2} \\
g \left( x \right) &= - x^{6} + \left( -\frac{a_{7}}{2} - \alpha_{1} \right) x^{5} + \left( - \frac{a_{6}}{2} + \frac{a_{7}^{2}}{8} - \frac{\alpha_{1}a_{7}}{2} - \alpha_{2} \right) x^{4} + \alpha_{3}x^{3} + \alpha_{4}x^{2} + \alpha_{5}x + \alpha_{6}.
\end{align*}
\end{proposition}

\begin{proof}
Let  $\infty_{+}$ and $ \infty_{-}$ be the two points at infinity  and  $\left[ D \right] \in J\left[ 3\right] \setminus 0$. By Riemann-Roch there exists a unique,  effective divisor $D_{0} = P_{1} + P_{2} + P_{3} $  such that 
\begin{center}
$ D \sim D_{0} - \infty_{+} - 2 \infty_{-}$
\end{center}
As $3D$ is principal,  $3D_{0} -3 \infty_{+} - 6 \infty_{-} = \text{ div} \left( h \right) $ , where $h$ is a rational function on $C$. Thus $h$ is in the Riemann-Roch space $\mathcal{L}( 3 \infty_{+} +  6 \infty_{-})$. As function on $C$, $x$ has poles of order exactly $1$ at both $\infty_{+}$ and $\infty_{-}$ and thus $1, x, x^2, x^3$ are elements of $\mathcal{L}( 3 \infty_{+} + 6 \infty_{-})$. Moreover, consider the functions $g_{1} = y - x^4$, $ g_{2} = xy - x^{5} - \frac{a_{7}}{2}x^{4}$ and $g_{3} = x^{2}y - x^{6} - \frac{a_{7}}{2}x^{5} + \left( -\frac{a_{6}}{2} + \frac{a_{7}^{2}}{8} \right) x^{4}$.  From our explicit model, we note that these have a pole of order exactly $4$, $5$ and $6$ resp. at $\infty_{-}$, and they all have have a pole of order at most $3$ at $\infty_{+}$ (accounting for the fact that $\infty_{+}$ is a possible zero of the $g_{i}$, with order determined by the coefficients), hence these functions contained in $ \mathcal{L}( 3 \infty_{+} + 6 \infty_{-})$. Moreover, by the Riemann-Roch theorem, the space has dimension  $7 $ and comparing orders at $ \infty_{-}$, we note that  $1, x, x^{2}, x^{3}, g_{1}, g_{2},g_{3}$ are all linearly independent and thus necessarily form a basis of the Riemann-Roch space. 

By possibly replacing $h$ by a scalar multiple,  $h$ will necessarily fall in one of the following four cases.
\begin{case2} 
Suppose $h$ is a polynomial in $x$ of degree $d \leq 3$.  Let $\theta_{1}, \ldots, \theta_{d}$ be the roots of $h$.  The divisor of zeros of $h$ is $3D_{0}$ since  $ \text{ div} \left( h \right) = 3D_{0} - 9 \infty$. We can also compute the divisor of zeros directly,  and find it to be 
\begin{center}
$\displaystyle \sum_{i=1}^{d} \left( \left( \theta_{i} , \sqrt{f\left( \theta_{i} \right)} \right) +  \left( \theta_{i} ,  - \sqrt{f\left( \theta_{i} \right)} \right) \right) $
\end{center}
The above divisor has degree at most $6$,  whilst $\text{deg}\left( 3D_{0} \right) =9$, and hence they cannot be equal. Thus $h$ cannot be a polynomial in $x$ of degree at most $3$.
\end{case2}

\begin{case2}
Suppose $h$ is of the form 
\begin{align*}
h &  = y - x^{4} + \alpha_{1}x^{3} + \alpha_{2}x^{2} + \alpha_{3}x + \alpha_{4}  \\ & =  \ y + g\left( x \right)
\end{align*}
for some $\alpha_{1}, \ldots, \alpha_{4} \in \overline{K}$, where $g\left( x \right) = - x^{4} + \alpha_{1}x^{3} + \alpha_{2}x^{2} + \alpha_{3}x + \alpha_{4} $. Let $\tilde{h} = -y + g\left( x \right)$.  Arguing as in the proof of the previous proposition, the divisor of zeros of $h$ is $3D_{0}$; and the divisor of zeros of $\tilde{h}$ is $3 \iota \left( D_{0} \right)$. The divisor of zeros of $h\tilde{h} \in K \left[ x \right]$ is $3 D_{0} + 3 \iota \left( D_{0} \right)$, and thus by Lemma 1, $h\tilde{h} \in K \left[ x \right]$ is necessarily a cube. 
We find that 
\begin{align*}
h\tilde{h} &= \left( y + g\left( x \right) \right) \left( -y + g \left( x \right) \right) \\
               &= -f\left( x \right) + g\left( x \right)^{2}
\end{align*} 
has degree at most 7, and hence it has degree 6 or 3 if it is indeed a cube.   Suppose $h\tilde{h}$ has degree 6, so $h\tilde{h} = q^{3}$ where $q \in K\left[ x \right]$ is a quadratic polynomial.  Let $\theta_{1}, \theta_{2}$ be the roots of $q$. Then the divisor of zeros of $h\tilde{h}$ is
\begin{center}
$\displaystyle 3  \sum_{i=1}^{2} \left( \left( \theta_{i}, \sqrt{f\left( \theta_{i} \right) } \right) + \left( \theta_{i} , - \sqrt{f \left( \theta_{i} \right) }\right) \right) $
\end{center} 
and by considering the degree of this divisor,  it cannot equal $3D_{0}$.  A very similar argument shows that the $\text{deg} \left( h \tilde{h} \right) = 3$ also leads to a contradiction. Thus $h$ cannot be of the stated form. 
\end{case2}
\begin{case2}
Suppose $h$ is of the form 
\begin{align*}
h &= xy - x^{5} - \frac{a_{7}}{2}x^{4} + \alpha_{1} \left( y - x^{4} \right) + \alpha_{2}x^{3} + \alpha_{3}x^{2} + \alpha_{4}x + \alpha_{5} \\
&= l\left( x \right)y + g \left( x \right)
\end{align*}
for some $\alpha_{1}, \ldots, \alpha_{5} \in \overline{K}$, where $l\left( x \right) = x + \alpha_{1}$, $g\left( x \right) = -x^{5} -\frac{a_{7}}{2}x^{4} - \alpha_{1}x^{4} + \alpha_{2}x^{3} + \alpha_{3}x^{2} + \alpha_{4}x + \alpha_{5}$. Let $\tilde{h} = -  l\left( x \right) y + g \left( x \right)$.  Arguing as before, $h\tilde{h} \in K \left[ x \right]$ is a cube. We find that 
\begin{align*}
h\tilde{h} &= \left( l\left( x \right)y + g\left( x \right) \right) \left( -l\left( x \right)y + g \left( x \right) \right) \\
               &= -l\left( x \right)^{2} f\left( x \right) + g\left( x \right)^{2}
\end{align*} 
has degree at most $8$, and hence it must have degree $3$ or $6$ if it is a cube.  As in case 2, both possible degrees lead to a contradiction.  Hence $h$ cannot be of the stated form. 
\end{case2}
\begin{case2}
Suppose $h$ is of the form 
\begin{align*}
h &= x^{2}y - x^{6} - \frac{a_{7}}{2}x^{5} + \left( -\frac{a_{6}}{2} + \frac{a_{7}^{2}}{8} \right) x^{4} +   \alpha_{1}\left(  xy - x^{5} - \frac{a_{7}}{2}x^{4} \right)  + \alpha_{2} \left( y - x^{4} \right)  \\ & \ \ \ \ + \alpha_{3}x^{3} + \alpha_{4}x^{2} + \alpha_{5}x + \alpha_{6} \\
&= l\left( x \right)y + g \left( x \right)
\end{align*}
for some $\alpha_{1}, \ldots, \alpha_{6} \in \overline{K}$, where $l\left( x \right) = x^{2} + \alpha_{1}x + \alpha_{2}$, $g\left( x \right) = -x^{6} - \frac{a_{7}}{2}x^{5}  - \left( -\frac{a_{6}}{2} + \frac{a_{7}^{2}}{8} \right)x^{4} + \alpha_{1} \left( - x^{5} - \frac{a_{7}}{2}x^{4} \right) - \alpha_{2} x^{4} + \alpha_{3}x^{3} + \alpha_{4}x^{2} + \alpha_{5}x + \alpha_{6}$.
Following previous arguments, set $\tilde{h} = -l\left( x \right)y + g\left( x \right)$, then by considering the divisor of zeros of $h\tilde{h}$ we find that $h\tilde{h} \in K \left[ x \right]$ must be a cube.  In general, 
\begin{align*}
h\tilde{h} &= \left( l\left( x \right)y + g\left( x \right) \right) \left( -l\left( x \right)y + g \left( x \right) \right) \\
               &= -l\left( x \right)^{2} f\left( x \right) + g\left( x \right)^{2}
\end{align*} 
has degree 9, and so it must be the cube of a degree 3 polynomial; and so there exist $\alpha_{7}, \ldots, \alpha_{10} \in \overline{K}$ such that 
\begin{center}
$ -l\left( x \right)^{2} f\left( x \right) + g\left( x \right)^{2} = \alpha_{7} \left( x^{3} + \alpha_{8}x^{2} \alpha_{9}x + \alpha_{10} \right)^{3}$
\end{center}
Thus such $h$ define 3-torsion points on $J$. 
\end{case2}
\end{proof}

As in the previous case, we equate coefficients in the expression 
\begin{align*}
&  \left( -x^{6} - \frac{a_{7}}{2}x^{5}  - \left( -\frac{a_{6}}{2} + \frac{a_{7}^{2}}{8} \right)x^{4} + \alpha_{1} \left( - x^{5} - \frac{a_{7}}{2}x^{4} \right) - \alpha_{2} x^{4} + \alpha_{3}x^{3} + \alpha_{4}x^{2} + \alpha_{5}x + \alpha_{6} \right)^{2}  \\ &  \ \ \ \ \ \ =  \alpha_{7} \left( x^{3} + \alpha_{8}x^{2} + \alpha_{9}x + \alpha_{10} \right)^{3}  + \left( x^{2} + \alpha_{1}x + \alpha_{2}\right)^{2}  f\left( x \right)
\end{align*}
to obtain 10 equations in $\alpha_{1}, \ldots, \alpha_{10}$, with $\left( \alpha_{1}, \ldots, \alpha_{6} \right)$ representing a 3-torsion point. 

\section{Explicit Computations of $3$-torsion points} \label{comps}
Let $e_{1}, \ldots ,   e_{10}$ be the equations in $\alpha_{1}, \ldots, \alpha_{10}$ defining a scheme of 3-torsion points as in the previous section. Due to the intricate nature of the solutions of this system, computing them via classical Gr\"{o}bner basis techniques
is extremely inefficient. Instead, we employ a two step strategy; we first approximate the solutions as complex tuples of numbers and then use lattice reduction to find precise expressions. This method is detailed in \cite{2torsion} and we give a short overview in this section. 

\subsection{Complex Approximations} The solutions of our system can be approximated as complex numbers using the Newton-Raphson method, which we briefly recall here, see  \cite[Page 298]{NA} for a detailed explanation. 

We view  $E = \left( e_{1}, \ldots, e_{10} \right)$ as a function $\mathbb{C}^{10} \longrightarrow \mathbb{C}^{10}$ and  write $dE$ for the Jacobian matrix of $E$. Suppose $\mathbf{x}_{0}$ is an approximate solution to $E$ with $dE\left( \mathbf{x}_{0} \right)$ invertible.  For $k \geq 1$, define 
\begin{center}
$\mathbf{x}_{k} = \mathbf{x}_{k-1} - dE\left( \mathbf{x}_{k-1} \right)^{-1}E\left( \mathbf{x}_{k-1} \right)$
\end{center}
Provided the initial approximation $\mathbf{x}_{0}$ is a good enough approximation,  the resulting sequence $\lbrace \mathbf{x}_{k} \rbrace_{k \geq 0 }$ converges to a root of $E$, with each iterate having increased precision. The initial complex approximations required by this method can be obtained using homotopy continuation and its implementation in \texttt{Julia} (see \cite{julia}). Concisely, homotopy continuation is method for approximating the solutions of a system of polynomial equations by deforming the solutions of a similar system.  More precisely, let $F$ be a system of $10$ polynomials in $\alpha_{1}, \ldots, \alpha_{n}$,  which has exactly $\text{deg} \left( E \right) =  \prod_{i} \text{deg} \left( e_{i} \right)$ solutions, where $\deg(e_{i})$ is the largest degree of a monomial of $e_{i}$ and consider the function $  H :  \mathbb{C}^{10} \times \left[ 0, 1 \right] \longrightarrow \mathbb{C}^{10} $ defined  by: 
\begin{center}
    $ H \left( \mathbf{x}, t \right) \  = \  \left( 1 - t \right) F\left( \mathbf{x} \right) + t E \left( \mathbf{x} \right)$
\end{center}
For a fixed $N \in \mathbb{N}$, define $H_{s}\left( \mathbf{x} \right) = H \left( \mathbf{x}, s/N  \right)$ for any $s \in \left[ 0, N \right] \cap \mathbb{N}$. Observe that when $N$ is large enough, the solutions of $H_{s} \left( \mathbf{x} \right)$ are good approximations of the solutions of $H_{s+1}\left( \mathbf{x} \right)$, and using the Newton-Raphson method we can increase their precision.  The solutions of $H_{0} \left( \mathbf{x} \right) = F\left( \mathbf{x} \right)$ are known, and they can be used to define solution paths to approximate solutions of $H_{N}\left( \mathbf{x} \right) = E \left( \mathbf{x} \right)$.

Homotopy Continuation  is implemented in the \texttt{Julia} package HomotopyContinuation.jl (see \cite{julia}). 

\subsection{Lattice Reduction and Algebraic Expressions}
Suppose $\left( a_{1}, \ldots,  a_{10} \right)$ is the complex approximation to a solution of 
 the system defined by $e_{1}, \ldots, e_{10}$, accurate to $k$ decimal places. In this subsection we describe how lattice reduction is used to find the minimal polynomial of the algebraic number approximated by $a_{i}$. This technique is standard, see for instance \cite[Section 2.7.2]{cohen} or \cite[Chapter 6.1]{smart}. Moreover, computing the minimal polynomial of a complex approximation is implemented in \texttt{MAGMA} by the command \texttt{MinimalPolynomial}, whose input is a complex approximation and the supposed degree of its minimal polynomial, and its output is a polynomial with integer coefficients and of the given degree, and which has a root that is well approximated by our first input. This implementation uses similar techniques to the method described in this subsection, but we found it to be slower, especially when the degree of the polynomial is large.

For a fixed i, let $\alpha = a_{i}$ and suppose that this approximates $\theta$, an algebraic number. Then, for some $d \in \mathbb{N}$ and $c_{0}, \ldots, c_{d} \in \mathbb{Z}$ 
\begin{center}
$c_{d}\alpha^{d} + \ldots + c_{1}\alpha + c_{0} = 0$.
\end{center}
Suppose $\theta \in \mathbb{R}$. Fix a constant $C \thicksim  10^{k}$ and let $\mathcal{L}_{k}$ be the lattice generated by the columns of the $\left( d+1 \right) \times \left( d + 1 \right)$ matrix 
\begin{center}
$A_{k} = \begin{pmatrix}
1 & \ldots & 0 & 0  \\
0 & \ldots & 0 & 0 \\
\vdots & \ddots & \vdots & \vdots \\
0 & \ldots & 1& 0 \\
\left[ C \theta^{d} \right]   & \ldots & \left[ C \theta \right]  & \left[ C \right]  

\end{pmatrix}
 = \left( v_{d} , \ldots, v_{1}, v_{0} \right)
$
\end{center}
where $\left[ a \right]$ denotes the integer part of $a$. As $c_{0}, \ldots, c_{d} \in \mathbb{Z}$
\begin{center}
$  \mathbf{c}_{k} = \begin{pmatrix}
c_{d} \\
\vdots \\
c_{1} \\
a
\end{pmatrix} = c_{d}v_{n} + \ldots + c_{0}v_{0}  \in \mathcal{L}_{k} 
$
\end{center}
where $a = c_{d}\left[ C \theta^{d} \right]  + \ldots + c_{1} \left[ C \theta \right]  + c_{0}  \left[ C \right]  $. We can recover $\mathbf{c}_{\infty} = \left( c_{d} , \ldots , c_{0} \right)$ from $\mathbf{c}_{k}$ by setting 
\begin{center}
$ c_{0} =  a - \left(  c_{d}\left[ C \theta^{d} \right]  + \ldots + c_{1} \left[ C \theta \right]  \right)  $
\end{center} 

Arguing as in \cite[Section 5.2]{2torsion}, for any $k \geq 1 $
\begin{center}
$
\vert \vert \mathbf{c}_{k} \vert \vert   = \sqrt{ c_{d_{\theta}}^{2} + \ldots + c_{1}^{2} + \gamma^{2} }  \le \sqrt{2} \  \vert \vert \ \mathbf{c}_{\infty} \vert \vert^{2} $
\end{center}
and hence for any precision $k$,  $\vert \vert \mathbf{c}_{k} \vert \vert$ is bounded by the fixed constant $ \sqrt{2} \vert \vert \ \mathbf{c}_{\infty} \vert \vert^{2} $. Hermite's theorem \cite[page 66]{smart} predicts the length of the shortest vector in $\mathcal{L}_{k}$ to be approximately  $d \left( \mathcal{L}_{k} \right)^{\frac{1}{d+1 }}$, where 
$d\left( \mathcal{L_{k}} \right)$ is the determinant of the lattice. Thus, heuristically, if  the shortest vector of $\mathcal{L}_{k}$ is much smaller than this bound, it is a good candidate for the vector $\mathbf{c}_{k}$.
To search for short vectors in $\mathcal{L}$ we use the \texttt{Magma} command \texttt{ShortestVectors} which given a lattice returns a sequence containing all vectors of the lattice which have the minimum non-zero norm, see \cite{shortvectors}. The \texttt{Magma} algorithm for computing short vectors first computes a reduced basis of the lattice using its efficient implementation of the LLL-algorithm, which is based on both Nguyen and Stehl\`{e}'s floating-point LLL algorithm \cite{NSLLL} and de Weger's exact integral algorithm \cite{dWLLL}; and then uses a closest vector algorithm, such as the one described by Fincke and Pohst \cite{FP}, to determine the shortest non-zero vector in $\mathcal{L}$. 

\begin{remark}
When the imaginary part of $\theta$ is not 0,  the same method can be used with $\mathcal{L}_{k}$ being generated by the columns of 

\begin{center}
$A_{k} = \begin{pmatrix}
1 & \ldots & 0 & 0   \\
0 & \ldots & 0 & 0   \\
\vdots & \ddots & \vdots & \vdots\\
0 & \ldots & 1 & 0 \\
\left[ C \text{Re}\left(  \theta^{d} \right)  \right]   & \ldots  &   \left[ C  \text{Re} \left( \theta \right)  \right] &  \left[ C \right]  \\
\left[ C \text{Im}\left(  \theta^{d} \right)  \right]   & \ldots   &   \left[ C  \text{Im} \left( \theta \right)  \right] &  0  
\end{pmatrix}
$
\end{center}

where $\text{Re} \left( \theta \right)$ and $\text{Im} \left( \theta \right)$ denote the real and imaginary parts of $\theta$.
\end{remark}

To summarise, the strategy for finding the coefficients of the minimal polynomial of $\alpha$ is as follows. 

\begin{itemize}
\item[1.]Choose $d$.
\item[2.] Define the lattice  $\mathcal{L}_{k}$.
\item[3.] In  $\mathcal{L}_{k}$ look for vectors which are shorter than, say  $1/1000 d \left(  \mathcal{L}_{k} \right)^{\frac{1}{d+1}} $.  If such a vector doesn't exists,  either increase the precision k and start again, or choose a different degree and start again. 
\item[4.] If such a vector exists, verify that $\theta$ is an approximate solution of the corresponding polynomial. If this is not the case,  choose a different degree and start again.
\end{itemize}

Once we have candidates for the minimal polynomials of all coordinates approximated by $(a_{1}, \ldots, a_{10}) $ we want to identify which roots of the minimal polynomials correspond to the point approximated. The simplest way of doing this is by factorising all polynomials and testing which combinations define solutions of the given system. However, even when the degrees of the minimal polynomials are fairly small, this can be very impractical. Instead, we use lattice reduction to search for algebraic relations between coefficients. 
For instance, if the minimal polynomial corresponding to $a_{2}$ has a root over the number field defined by the polynomial corresponding to $a_{1}$, we suspect that for some $b_{0}, \ldots, b_{d_{1}} \in \Z $
\begin{center}
    $-b_{d_{1}}x_{2} = b_{d_{1} -1}x_{1}^{d_{1} -1} + \ldots + b_{1}x_{1} + b_{0}$
\end{center}
where $x_{1}, x_{2}$ are the algebraic numbers approximated by $a_{1}, a_{2}$ and $d_{1}$ is the degree of the minimal polynomial of $f_{1}$. Similar to the method used to search for minimal polynomial, such integers can be found by searching for short vectors in the lattice defined using the approximations. 
If we suspect that no relations as above exist, we can use the same method to look for higher order relations that could help to identify the corresponding point. 

The correctness of the minimal polynomials and coefficient relations computed can be checked by verifying that the  points they define are solutions of the initial system of equations $e_{1}, \dots, e_{10}$.

\section{Example: $J_{0}(40)[3]$} \label{eg1}
\allowdisplaybreaks
Our method is especially practical when $J[3]$ is defined over a number field of small degree. TO demonstate this, we compute the 3-torsion subgroup of the modular Jacobian $J_{0} \left( 40 \right)$.The \texttt{MAGMA} code in the computations presented in this section can be found in the online repository:
\begin{center}
\href{https://github.com/ElviraLupoian/3TorsionOfGenus3HypCurves}{https://github.com/ElviraLupoian/3TorsionOfGenus3HypCurves}
\end{center}

The model of the modular curve $X_{0} \left( 40 \right)$ given by \texttt{Magma} is:
\begin{center}
$ y^{2} + ( -x^4 - 1  ) y =  2x^6 - x^4 + 2x^2$.
\end{center}
Completing the square gives a model of the form required by Section 2,
\begin{center}
$ y^{2} = x^8 + 8x^6 -2x^4 + 8x^2 + 1 $.
\end{center}
The scheme of 3-torsion points is defined by 10 equations 
\begin{align*}
& \alpha_{2}^2 - \alpha_{6}^2 - \alpha_{9}^3\alpha_{10} , \\
  &  2\alpha_{1}\alpha_{2} - 2\alpha_{5}\alpha_{6} - 3\alpha_{8}\alpha_{9}^2\alpha_{10}, \\   
  & \alpha_{1}^2 + 8\alpha_{2}^2 + 2\alpha_{2} - 2\alpha_{4}\alpha_{6} - \alpha_{5}^2 - 3\alpha_{7} \alpha_{9}^2\alpha_{10} - 
        3\alpha_{8}^2\alpha_{9}\alpha_{10} , \\    
        &  16\alpha_{1}\alpha_{2} + 2\alpha_{1} - 2\alpha_{3}\alpha_{6} - 2\alpha_{4}\alpha_{5} - 6\alpha_{7}\alpha_{8}\alpha_{9}\alpha_{10} -
       \alpha_{8}^3\alpha_{10} - 3\alpha_{9}^2\alpha_{10}, \\
 &   8\alpha_{1}^2 - 2\alpha_{2}^2 + 2\alpha_{2}\alpha_{6} + 16\alpha_{2} - 2\alpha_{3}\alpha_{5} - \alpha_{4}^2 + 8\alpha_{6} 
        - 3\alpha_{7}^2\alpha_{9}\alpha_{10} - 3\alpha_{7}\alpha_{8}^2\alpha_{10} - 6\alpha_{8}\alpha_{9}\alpha_{10} + 1 , \\
  &  -4\alpha_{1}\alpha_{2} + 2\alpha_{1}\alpha_{6} + 16\alpha_{1} + 2\alpha_{2}\alpha_{5} - 2\alpha_{3}\alpha_{4} + 8\alpha_{5} - 
        3\alpha_{7}^2\alpha_{8}\alpha_{10} - 6\alpha_{7}\alpha_{9}\alpha_{10} - 3\alpha_{8}^2\alpha_{10} , \\
   & -2\alpha_{1}^2 + 2\alpha_{1}\alpha_{5} + 8\alpha_{2}^2 + 2\alpha_{2}\alpha_{4} - 4\alpha_{2} - \alpha_{3}^2 + 8\alpha_{4} 
        + 2\alpha_{6} - \alpha_{7}^3\alpha_{10} - 6\alpha_{7}\alpha_{8}\alpha_{10} - 3\alpha_{9}\alpha_{10} + 8 , \\
 &   16\alpha_{1}\alpha_{2} + 2\alpha_{1}\alpha_{4} - 4\alpha_{1} + 2\alpha_{2}\alpha_{3} + 8\alpha_{3} + 2\alpha_{5} - 
        3\alpha_{7}^2\alpha_{10} - 3\alpha_{8}\alpha_{10} , \\
   & 8\alpha_{1}^2 + 2\alpha_{1}\alpha_{3} + 8\alpha_{2} + 2\alpha_{4} - 3\alpha_{7}\alpha_{10} - 18, \\
 &   8\alpha_{1} + 2\alpha_{3} - \alpha_{10}, \\
\end{align*}

where the 3-torsion points are classes of divisors of the form $\frac{1}{3} \text{div} \left( h \right)$, 
\begin{center}
$ h   = x^2y - x^6 -4x^4  + \alpha_{1} \left(   xy - x^5  \right) + \alpha_{2} \left( y - x^4 \right)  + \alpha_{3}x^3  + \alpha_{4}x^{2} + \alpha_{5}x + \alpha_{6} $.
\end{center}
We find that $J_{0}\left(40 \right) \left[ 3 \right]$ can be generated using 3 Galois orbits,  2 consisting of 6 points each, and 1 consisting of 8 points; see Table \ref{tab:my_label}. For each orbit, we give the minimal polynomial of $\alpha_{1}$ and expressions for $\alpha_{2}, \ldots, \alpha_{6}$ in terms of $\alpha_{1}$.

\begin{longtable}{| c |} 
\hline
    $ u^6 + 4u^4 - 8u^2 + 12$ \\
\hline
$ \alpha_{1} = u $\\
$ \alpha_{2} = u + 1 $ \\
$ \alpha_{3} = \left( -1/9 \right)(u^5 + u^3 + 16u  + 18) $ \\
$ \alpha_{4} = \left(-1/3 \right)(u^5 + u^3 + 4u  + 3) $ \\
$ \alpha_{5} = \left(-1/3 \right)(u^5 + u^3 + u  - 6) $ \\
$ \alpha_{6} =  \left(-1/9 \right)(u^5 + u^3 + 7u  + 9) $\\
\hline 
\hline 
 $u^6 - 6u^5 + 4u^4 + 24u^3 + 256u^2 - 576u + 324 $\\
 \hline 
 $ \alpha_{1} = u $\\
$ \alpha_{2} = \left(  1/198 \right) (-u^4 + 4u^3 + 58u^2 - 124u + 126 ) $   \\ 
$ \alpha_{3} =  \left(-1/99 \right)(u^4 - 4u^3 - 58u^2 + 322u  + 468)$  \\
$ \alpha_{4} = \left( -1/99 \right)(u^4 - 4u^3 - 58*u^2 - 74u + 765)$ \\
$ \alpha_{5}  =  \left(-1/99 \right)(u^4 - 4u^3 - 58u^2 - 173u  + 468)$ \\
$ \alpha_{6} = \left(1/198 \right)(u^4 - 4u^3 - 58u^2 + 520u  - 522 )$ \\
\hline 
\hline 
$ u^8 - 126u^4 - 648u^2 - 1323 $\\
\hline 
$ \alpha_{1} = u $ \\
$ \alpha_{2} = -1 $ \\
$ \alpha_{3} =  \left( 1/189 \right)(u^7 - 63u^3 - 648u )$ \\
$ \alpha_{4} =  3 $ \\
$ \alpha_{5}  = -u $  \\
$ \alpha_{6} = 1 $\\
\hline 
 \caption{Generators for $J_{0}(40)[3]$}
    \label{tab:my_label}
\end{longtable}

The field of definition of the 3-torsion subgroup is the degree 48 number field defined by 
\begin{center}
$ x^{48} - 22x^{47} + 220x^{46} - 1298x^{45} + 4840x^{44} - 10758x^{43} + 7848x^{42} + 30564x^{41} - 90644x^{40} - 54378x^{39} + 983934x^{38} - 3228430x^{37} + 6037118x^{36} - 6706868x^{35} + 3859158x^{34} - 6290682x^{33} + 41469355x^{32} - 151827480x^{31} + 375328308x^{30} - 727099012x^{29} + 1204881284x^{28} - 1812362612x^{27} + 2558319144x^{26} - 3402905364x^{25} + 4192192588x^{24} - 4669768140x^{23} + 4602283152x^{22} - 3939374364x^{21} + 2873125672x^{20} - 1738390504x^{19} + 830314684x^{18} - 275496188x^{17} + 30094447x^{16} + 31178478x^{15} - 22364652x^{14} + 5362086x^{13} + 2307708x^{12} - 2995626x^{11} + 1676724x^{10} - 615660x^9 + 121728x^8 + 25686x^7 - 31194x^6 + 9162x^5 + 1458x^4 - 2088x^3 + 738x^2 - 126x + 9$ 
\end{center}

\section{Local Conductor Exponent at 2}
\label{cond}
Throughout this section,  $C/K$ is a smooth, projective curve defined over $K$, a finite extension of $\mathbb{Q}_{p}$.

Let $J$ be the Jacobian variety associated to $C$, $T = T_{l}J$ the $l$-adic Tate module and $V = V_{l}J = T \otimes_{\mathbb{Z}_{l}} \mathbb{Q}_{l}$ the associated $l$-adic representation, where $l$ is any prime different from $p$.  The \textit{conductor exponent} of such a representation,  as defined in \cite{DokDor} and \cite{ulmer}, is 
\begin{center}
$n = \displaystyle \int_{-1}^{\infty} \text{codim}  V^{G_{K}^{u}} \ du$
\end{center}
where $G_{K} = \Gal \left( \overline{K} / K \right)$ is the absolute Galois group of $K$ and $ \lbrace G_{K}^{u} \rbrace_{u \ge -1}$ denote the ramification groups of $G_{K}$ in upper numbering, classically defined as 
\begin{center}
    $G_{i} = \{ \sigma \in G \ \vert \sigma (\alpha) \equiv \alpha \mod{\pi^{i+1}} \ \text{for all} \alpha \in \mathcal{O}_{K} \}$
\end{center}
for all $i \ge -1$, where $\pi$ is a uniformizer. Note that $G_{-1} = G$ is the entire group and $G_{0} = I$ is simply the inertia subgroup of $G$.
The tame and wild parts are defined as 
\begin{align*}
n_{\text{tame}} & =  \displaystyle \int_{-1}^{0} \text{codim}  V^{G_{K}^{u}} \ du \\
n_{\text{wild}} & =  \displaystyle \int_{0}^{\infty} \text{codim}  V^{G_{K}^{u}} \ du.
\end{align*}

This definition is independent of the choice of prime $l$, see \cite{ulmer}. Our approach is as in \cite{DokDor}, namely for $p=2$ we take $l =3$ and use the explicit generators of the 3-torsion subgroup.
\subsection{Tame Conductor}
From the above, the tame part of the conductor can be computed as 
\begin{center}
$n_{\text{tame}} = 2g - \text{dim}V_{l}J^{I}$
\end{center}
where $I \leqslant G_{K}$ is the inertia subgroup. Notably, in the special case where $J[l]^{I} = \{ 0 \}$, we  deduce that $n_{\text{tame}}  = 2g$. In general, the tame part of the conductor can be computed from the special fibre of the regular model of $C$. More precisely, we compute:
\begin{itemize}
\item the abelian part $a$,  equal to the sum of the genera of all components;
\item the toric part $t$,  equal to the number of loops in the dual graph of $C$.
\end{itemize}

Then, the tame part of the exponent is equal to $2g -2a -t$, see \cite[Chapter 9]{NeronModels} for details. For plane curves, regular models can often by computed using the method described in \cite{ModelsoverDVRs}. In general, this is a notoriously challenging computational problem. 
\subsection{Wild Conductor} For $u \ge 0$, $G_{K}^{u}$ is pro-p and $\text{codim} V^{G_{K}^{u}} = \text{codim} \bar{V}^{G_{K}^{u}} = \text{codim} J\left[ l \right]^{G_{K}^{u}}$, see \cite{ulmer}. Thus we replace $G_{K}$ by $G = \Gal \left( K\left( J[l] \right) / K \right)$, and the wild part of the conductor is simply 
\begin{center}
$ n_{\text{wild}} =  \displaystyle \int_{0}^{\infty} \text{codim}  J[l]^{G^{u}} \ du$,
\end{center}
Alternatively,  using the definition of $G^{u}$ and $G_{u}$,  the ramification groups in upper numbering and lower numbering respectively, we find 
\begin{center}
$ n_{\text{wild}} =  \displaystyle \int_{0}^{\infty} \frac{\text{codim}  J[l]^{G_{u}}}{[ G_{0} : G_{u} ]} \ du = \displaystyle \sum_{k=1}^{\infty} \frac{\text{codim}  J[l]^{G_{k}}}{[ G_{0} : G_{k} ]}  $.
\end{center}

In our setting, when $p=2$ we take $l=3$ and using our presentation of $J\left[ 3 \right]$, the ramification groups $G_{u}$ and their action on $J\left[ 3 \right]$ are completely explicit.

\subsection{Example continued}
The 3-torsion subgroup $J_{0}\left( 40 \right)[3]$ is defined over a degree $48$ number field defined by the polynomial $f$, stated in Section \ref{eg1}. This polynomial remains irreducible over $\mathbb{Q}_{2}$, and defines a degree $48$ Galois extension of $\mathbb{Q}_{2}$, which we denote by $L$. Let $G = \Gal \left( L/ \mathbb{Q}_{2} \right)$. 
Using \texttt{Magma} we verify that $G \cong \GL_{2}(\mathbb{F}_{3})$ and find explicit generators for $G$ 
\begin{center}
    $ G = \langle \tau_{1}, \tau_{2}, \beta, \sigma_{1}, \sigma_{2} \rangle$
\end{center}
where $\tau_{i}$ have order $2$, $\beta$ has order $3$ and $\sigma_{j}$ have order 4.
Moreover, using our explicit description we find 2 divisors in the first orbit, 2 divisors in the second orbit and 2 divisors in the third orbit, which together form a basis for $J[3]$, see the \texttt{Galact.m} file of the repository for details. We explicitly compute the action of the above geneators on this basis and we find 
\begin{center}
  $\tau_{1} = \begin{psmallmatrix}
      1 & 0 & 0 & 0 & 0 & 0 \\ 0 & 1 & 0 & 0 & 0 & 0 \\ 0 & 0 & 1 & 0 & 0 & 0 \\ 0 & 0 & 0 & 1 & 0 & 0 \\ 0 & 0 & 0 & 0 & 2 & 0 \\ 0 & 0 & 0 & 0 & 0 & 2 
  \end{psmallmatrix};$  
   $\tau_{2} = \begin{psmallmatrix}
      0 & 2 & 1 & 1 & 0 & 0 \\ 2 & 0 & 2 & 2 & 0 & 0 \\ 0 & 0 & 0 & 2 & 0 & 0 \\ 0 & 0 & 2 & 0 & 0 & 0 \\ 0 & 0 & 0 & 0 & 0 & 1 \\ 0 & 0 & 0 & 0 & 1 & 0 
  \end{psmallmatrix}; $  
   $\sigma_{1} = \begin{psmallmatrix}
      1 & 0 & 0 & 0 & 0 & 0 \\ 0 & 1 & 0 & 0 & 0 & 0 \\ 0 & 0 & 1 & 0 & 0 & 0 \\ 0 & 0 & 0 & 1 & 0 & 0 \\ 0 & 0 & 0 & 0 & 0 & 1 \\ 0 & 0 & 0 & 0 & 2 & 0 
  \end{psmallmatrix}$; 
   $\sigma_{2} = \begin{psmallmatrix}
      1 & 0 & 0 & 0 & 0 & 0 \\ 0 & 1 & 0 & 0 & 0 & 0 \\ 0 & 0 & 1 & 0 & 0 & 0 \\ 0 & 0 & 0 & 1 & 0 & 0 \\ 0 & 0 & 0 & 0 & 1 & 2  \\ 0 & 0 & 0 & 0 & 2 & 2 
  \end{psmallmatrix}$ ;
   $\beta = \begin{psmallmatrix}
      0 & 2 & 0 & 0 & 0 & 0 \\ 1 & 2 & 0 & 0 & 0 & 0 \\ 0 & 0 & 0 & 2 & 0 & 0 \\ 0 & 0 & 1 & 2 & 0 & 0 \\ 0 & 0 & 0 & 0 & 2 & 2 \\ 0 & 0 & 0 & 0 & 1 & 0 
  \end{psmallmatrix}$.
  \end{center}
Moreover, we compute the ramification group and found them to be as follows:
\begin{align*}
& G_{0} = \langle \tau_{1}, \beta, \sigma_{1}, \sigma_{2} \rangle \ \text{and} \ \vert G_{0} \vert = 24; \\
& G_{1} =  \langle \tau_{1}, \sigma_{1}, \sigma_{2} \rangle \ \text{and} \ \vert G_{1} \vert = 8;\\
& G_{2} = G_{3}  =  \langle \tau_{1} \rangle \ \text{and} \ \vert G_{2} \vert = \vert G_{3} \vert = 2; \\
& G_{n} = 1 \ \text{for all } \ n \ge 4.
\end{align*}
We also verify that $G_{0} \cong \SL_{2}(\mathbb{F}_{3})$ and $G_{1} \cong Q_{8}$.
Taking Galois invariants we get:
\begin{center}
 $ J_{0}\left( 40 \right)[3]^{G_{1}} \cong \left( \mathbb{Z} / 3\mathbb{Z} \right)^{4}$  and $ J_{0}\left( 40 \right)[3]^{G_{2}} \cong \left( \mathbb{Z} / 3\mathbb{Z} \right)^{4}$   
\end{center}
and thus  $n_{\text{wild}} = 2/3 + 2/12 + 2/12 = 1$. See the \texttt{X040wildcondexp.m} file of the repository for details of this calculation. 
We also note that $J_{0}(40)[3]^{G_{0}} \cong \left( \Z / 3 \Z \right)^{2}$, and thus the tame part of the conductor cannot be deduce from the $3$-torsion. Using \cite{edixhoven1990minimal} we compute the minimal regular model of $X_{0}(40)$ over $\Z_{2}$, and note that its special fibre has $8$ components, each of genus $0$ and no loops in the dual graph, and thus $n_{2, \text{tame}} = 6$, and so $n_{2} =7$.
\bibliographystyle{abbrv}
\bibliography{ref}

\end{document}